\documentclass{amsart}

\usepackage{amsfonts}
\usepackage{amsmath}
\usepackage{amssymb}
\usepackage{amscd}
\usepackage{mathrsfs}  
\usepackage{amsbsy}
\usepackage{amsthm}
\usepackage{graphicx}
\usepackage{fancyhdr}
\usepackage{epsfig}
\usepackage{color}
\usepackage{xy}

\usepackage{CJKutf8}
\usepackage{inputenc}
\usepackage[encapsulated]{CJK}
\usepackage[CJK, overlap]{ruby}

\usepackage[OT2,OT1]{fontenc}  
\newcommand\cyr{%
\renewcommand\rmdefault{wncyr}%
\renewcommand\sfdefault{wncyss}%
\renewcommand\encodingdefault{OT2}%
\normalfont \selectfont} \DeclareTextFontCommand{\textcyr}{\cyr}

\newcommand{\be}{\begin{equation}}
\newcommand{\ee}{\end{equation}}

\newcommand{\C}{\mathbb{C}}

\newcommand{\bH}{\mathbb{H}}
\newcommand{\K}{\mathbb{K}}
\newcommand{\N}{\mathbb{N}}

\newcommand{\R}{\mathbb{R}}

\newcommand{\X}{\mathbb{X}}
\newcommand{\Y}{\mathbb{Y}}
\newcommand{\Z}{\mathbb{Z}}

\newcommand{\bS}{\boldsymbol{S}}

\newcommand{\cW}{\mathcal{W}}

\newcommand{\st}{\,\vert\,}

\newcommand{\bmu}{\boldsymbol{\mu}} 
\newcommand{\blambda}{\boldsymbol{\lambda}} 
\newcommand{\by}{\boldsymbol{y}} 
\newcommand{\cF}{\mathcal{F}}

\newcommand{\ff}{\mathfrak{f}}

\newcommand{\oX}{\overline{X}}
\newcommand{\oXi}{\overline{X}_i}
\newcommand{\oPhi}{\overline{\Phi}}

\newcommand{\of}{\overline{f}}
\newcommand{\og}{\overline{g}}

\newcommand{\ou}{\overline{u}}
\newcommand{\ox}{\overline{x}}
\renewcommand{\rmdefault}{cmr} 
\renewcommand{\sfdefault}{cmr} 
\newtheorem{theorem}{Theorem}
\theoremstyle{plain}

\newtheorem{definition}{Definition}
\newtheorem{example}{Example}

\newtheorem{lemma}{Lemma}

\newtheorem{proposition}{Proposition}
\newtheorem{remark}{Remark}

\numberwithin{equation}{section}

\newcommand{\loc}{\mathrm{loc}}

\begin{document}

\title[Local Fractal Functions and Function Spaces]{Local Fractal Functions and Function Spaces}
\author{Peter R. Massopust}
\address{Centre of Mathematics, Research Unit M6, Technische Universit\"at M\"unchen, Boltzmannstrasse 3, 85747
Garching b. M\"unchen, Germany, and Helmholtz Zentrum M\"unchen,
Ingolst\"adter Landstrasse 1, 85764 Neuherberg, Germany}
\email{peter.massopust@helmholtz-muenchen.de, massopust@ma.tum.de}

\begin{abstract}
We introduce local iterated function systems and present some of their basic properties. A new class of local attractors of local iterated function systems, namely local fractal functions, is constructed. We derive formulas so that these local fractal functions become elements of various function spaces, such as the Lebesgue spaces $L^p$, the smoothness spaces $C^n$,  the homogeneous H\"older spaces $\dot{C}^s$, and the Sobolev spaces $W^{m,p}$.
\end{abstract}

\maketitle
\section{Introduction}\label{sec1}
Iterated function systems, for short IFSs, are a powerful means for describing fractal sets and for modeling or approximating natural objects. IFSs were first introduced in \cite{BD,Hutch} and subsequently investigated by numerous authors. Within the fractal image compression community a generalization of IFSs was proposed in \cite{barnhurd} whose main purpose was to obtain efficient algorithms for image coding. 

In \cite{BHM1}, this generalization of a traditional IFS, called a local IFS, was reconsidered but now from the viewpoint of approximation theory and from the standpoint of developing computationally efficient numerical methods based on fractal methodologies. In the current paper, we continue this former exploration of local IFSs and consider a special class of attractors, namely those that are the graphs of functions. We will derive conditions under which such local fractal functions are elements of certain function spaces which are important in harmonic analysis and numerical mathematics.

The structure of this paper is as follows. We present the traditional IFSs in Section \ref{sec2} in a more general and modern setting and state some of their properties. Section \ref{sec3} introduces local IFSs and discusses some characteristics of this newly rediscovered concept. Local fractal functions and their connection to local IFSs are investigated in Section \ref{sec4}. In Section \ref{sec5} we briefly consider tensor products of local fractal functions. Local fractal functions in Lebesgue spaces are presented in Section \ref{sec6}, in smoothness and H\"older spaces in Section \ref{sec7}, and in Sobolev spaces in Section \ref{sec8}.
\section{Iterated Function Systems}\label{sec2}
In this section, we introduce the traditional IFS and highlight some of its fundamental properties. For more details and proofs, we refer the reader to \cite{B,BD,bm,Hutch} and the references stated therein. 

Throughout this paper, we use the following notation. The set of positive integers is denoted by $\mathbb{N} := \{1, 2, 3, \ldots\}$, the set of nonnegative integers by $\N_0 = \N\cup\{0\}$, and the ring of integers by $\Z$. We denote the closure of a set $S$ by $\overline{S}$ and its interior by $\overset{\circ}{S}$. In the following, $(\mathbb{X},d_X)$ always denotes a complete metric space with metric $d_{\mathbb{X}}$.

\begin{definition}
Let $N\in\mathbb{N}$. If $f_{n}:\mathbb{X}\rightarrow\mathbb{X}$,
$n=1,2,\dots,N,$ are continuous mappings, then $\mathcal{F} :=\left(
\mathbb{X};f_{1},f_{2},...,f_{N}\right)  $ is called an \textbf{iterated
function system} (IFS).
\end{definition}

By a slight abuse of notation and terminology, we use the same symbol $\mathcal{F}$ for the
IFS, the set of functions in the IFS, and for the following set-valued mapping defined on the class of all subsets $2^\X$ of $\X.$ Define $\mathcal{F}:2^{\mathbb{X}}\rightarrow 2^{\mathbb{X}}$ by
\[
\mathcal{F}(B) := \bigcup_{f\in\mathcal{F}}f(B), \quad B\in 2^\X.
\]

Denote by $\mathbb{H=H(\X)}$ the hyperspace of all nonempty compact subsets of $\mathbb{X}$. The hyperspace $(\bH,d_\bH)$ becomes a complete metric space when endowed with the Hausdorff metric $d_{\bH}$ (cf. \cite{Engel})
\[
d_\bH (A,B) := \max\{\max_{a\in A}\min_{b\in B} d_X (a,b),\max_{b\in B}\min_{a\in A} d_X (a,b)\}.
\]
Since $\mathcal{F}\left(  \mathbb{H}\right)  \subset\mathbb{H}$, we can also treat $\mathcal{F}$ as a mapping $\mathcal{F}:\mathbb{H} \rightarrow \mathbb{H}$. When
$U\subset\mathbb{X}$ is nonempty, we may write $\mathbb{H}(U)=\mathbb{H(X)}%
\cap2^{U}$. We denote by $\left\vert \mathcal{F}\right\vert $ the number of
distinct mappings in $\mathcal{F}$.

A metric space $\mathbb{X}$ is termed \textbf{locally compact} if every point of $\X$ has a neighborhood that contains a compact neighborhood. The following information, a proof of which can be found in \cite{bm}, is foundational.
\begin{theorem}
\label{ctythm}
\begin{itemize}
\item[(i)] If $(\mathbb{X},d_{\mathbb{X}})$ is compact then $(\mathbb{H}%
,d_{\mathbb{H}})$ is compact.

\item[(ii)] If $(\mathbb{X},d_{\mathbb{X}})$ is locally compact then $(\mathbb{H}%
,d_{\mathbb{H}})$ is locally compact.

\item[(iii)] If $\mathbb{X}$ is locally compact, or if each $f\in\mathcal{F}$ is
uniformly continuous, then $\mathcal{F}:\mathbb{H\rightarrow H}$ is continuous.

\item[(iv)] If $f:\mathbb{X\rightarrow}\mathbb{X}$ is a contraction mapping for each
$f\in\mathcal{F}$, then $\mathcal{F}:\mathbb{H\rightarrow H}$ is a contraction mapping.
\end{itemize}
\end{theorem}
\noindent
For $B\subset\mathbb{X}$, let $\mathcal{F}^{k}(B)$ denote the $k$-fold
composition of $\mathcal{F}$, i.e., the union of $f_{i_{1}}\circ f_{i_{2}%
}\circ\cdots\circ f_{i_{k}}(B)$ over all finite words $i_{1}i_{2}\cdots i_{k}$
of length $k.$ Define $\mathcal{F}^{0}(B) := B.$

\begin{definition}
\label{attractdef}A nonempty compact set $A\subset\mathbb{X}$ is said to be an
\textbf{attractor} of the IFS $\mathcal{F}$ if
\begin{itemize}
\item[(i)] $\mathcal{F}(A)=A$, and if

\item[(ii)] there exists an open set $U\subset\mathbb{X}$ such that $A\subset U$ and
$\lim_{k\rightarrow\infty}\mathcal{F}^{k}(B)=A,$ for all $B\in\mathbb{H(}U)$,
where the limit is in the Hausdorff metric.
\end{itemize}
The largest open set $U$ such that $\mathrm{(ii)}$ is true is called the \textbf{basin of
attraction} (for the attractor $A$ of the IFS $\mathcal{F}$).
\end{definition}

Note that if $U_1$ and $U_2$ satisfy condition $\mathrm{(ii)}$ in Definition 2 for the same attractor $A$ then so does 
$U_1 \cup U_2$. We also remark that the invariance condition $\mathrm{(i)}$ is not needed; it follows from $\mathrm{(ii)}$ for $B := A$.

We will use the following observation \cite[Proposition 3 (vii)]{lesniak},
\cite[p.68, Proposition 2.4.7]{edgar}.

\begin{lemma}
\label{intersectlemma}Let $\left\{  B_{k}\right\}  _{k=1}^{\infty}$ be a
sequence of nonempty compact sets such that $B_{k+1}\subset B_{k}$, for all
$k\in\N$. Then $\cap_{k\geq1}B_{k}=\lim_{k\rightarrow\infty}B_{k}$ where
convergence is with respect to the Haudorff metric $d_\bH$.
\end{lemma}

The next result shows how one may obtain the attractor $A$ of an IFS. For the proof, we refer the reader to \cite{bm}. Note that we do not assume that the functions in the IFS $\cF$ are contractive.

\begin{theorem}
\label{attractorthm}Let $\mathcal{F}$ be an IFS with attractor $A$ and basin
of attraction $U.$ If the map $\mathcal{F}:\mathbb{H(}U)\mathbb{\rightarrow H(}U)$ is
continuous then%
\[
A=\bigcap\limits_{K\geq1}\overline{\bigcup_{k\geq K}\mathcal{F}^{k}(B)},%
\quad\text{ for all }B\subset U\text{ such that }\overline{B}\in
\mathbb{H(}U)\text{.}%
\]

\end{theorem}

The quantity on the right-hand side here is sometimes called the
\textbf{topological upper limit }of the sequence $\left\{ \cF^{k}(B)\st k\in \N\right\}$. (See, for instance, \cite{Engel}.) 

A subclass of IFSs is obtained by imposing additional conditions on the functions that comprise the IFS. The definition below introduces this subclass.

\begin{definition}
An IFS $\cF = (\X; f_1, f_2, \ldots, f_N)$ is called \textbf{contractive} if there exists a metric $d^*$ on $\X$, which is equivalent to $d$, such that each $f\in \cF$ is a contraction with respect to the metric $d^*$, i.e., there is a constant $c \in [0, 1)$
such that 
$$
d^*(f(x_1), f(x_2)) \leq c\,d(x_1, x_2),
$$
for all $x_1, x_2 \in \X$. 
\end{definition}
By item $\mathrm{(iv)}$ in Theorem 1, the mapping 
$\cF : \mathbb{H} \to \mathbb{H}$ is then also contractive on the complete metric space $(\mathbb{H}, d_{\mathbb{H}})$, and thus possesses a unique attractor $A$. This attractor satisfies the \textbf{self-referential equation}
\be\label{self}
A = \cF(A) = \bigcup_{f\in\mathcal{F}}f(A).
\ee
In the case of a contractive IFS, the basin of attraction for $A$ is $\X$ and the attractor can be computed via the following procedure: Let $K_0$ be any set in $\bH(\X)$ and consider the sequence of iterates
\[
K_m := \cF(K_{m-1}) = \cF^m (K_0), \quad m\in \N.
\]  
Then $K_m$ converges in the Hausdorff metric to the attractor $A$ as $m\to\infty$, i.e., $d_\bH(K_m, A) \to 0$ as $m\to\infty$.

For the remainder of this paper, the emphasis will be on contractive IFSs, respectively, contractive local IFSs. We will see that the self-referential equation \eqref{self} plays a fundamental role in the construction of fractal sets and in the determination of their geometric and analytic properties.

\section{From IFS to Local IFS}\label{sec3}
The concept of \textit{local} iterated function system is a generalization of an IFS as defined above and was first introduced in \cite{barnhurd} and reconsidered in \cite{BHM1}. In what follows, $N\in\N$ always denotes a positive integer and $\N_N := \{1, \ldots, N\}$.

\begin{definition}\label{localIFS}
Suppose that $\{\X_i \st i \in \N_N\}$ is a family of nonempty subsets of a metric space $\X$. Further assume that for each $\X_i$ there exists a continuous mapping $f_i: \X_i\to\X$, $i\in \N_N$. Then $\cF_{\loc} := \{\X; (\X_i, f_i)\st i \in \N_N\}$ is called a \textbf{local iterated function system} (local IFS).
\end{definition}

Note that if each $\X_i = \X$, then Definition \ref{localIFS} coincides with the usual definition of a standard (global) IFS on a complete metric space. However, the possibility of choosing the domain for each continuous mapping $f_i$ different from the entire space $X$ adds additional flexibility as will be recognized in the sequel.

\begin{definition}
A local IFS $\cF_{\loc}$ is called \textbf{contractive} if there exists a metric $d^*$ equivalent to $d$ with respect to which all functions $f\in \cF_{\loc}$ are contractive (on their respective domains).
\end{definition}
\noindent
With a local IFS we associate a set-valued operator $\cF_\loc : 2^\X \to 2^\X$ by setting
\be\label{hutchop}
\cF_\loc(S) := \bigcup_{i=1}^N f_i (S\cap \X_i).
\ee

By a slight abuse of notation, we use the same symbol for a local IFS and its associated operator.

\begin{definition}
A subset $A\in 2^\X$ is called a \textbf{local attractor} for the local IFS $\{\X; (\X_i, f_i)\st i \in \N_N\}$ if
\be\label{attr}
A = \cF_\loc (A) = \bigcup_{i=1}^N f_i (A\cap \X_i).
\ee
\end{definition}
In \eqref{attr} we allow for $A\cap \X_i$ to be the empty set. Thus, every local IFS has at least one local attractor, namely $A = \emptyset$. However, it may also have many distinct ones. In the latter case, if $A_1$ and $A_2$ are distinct local attractors, then $A_1\cup A_2$ is also a local attractor. Hence, there exists a largest local attractor for $\cF_\loc$, namely the union of all distinct local attractors. We refer to this largest local attractor as {\em the} local attractor of a local IFS $\cF_\loc$.

\begin{remark}
There exists an alternative definition for \eqref{hutchop}. We could consider the mappings $f_i$ as defined on all of $\X$ in the following sense: For any $S\in 2^\X$, let
\[
f_i (S) := \begin{cases} f_i (S\cap \X_i), & S\cap \X_i\neq \emptyset;\\ \emptyset, & S\cap \X_i = \emptyset,\end{cases}  \qquad i\in \N_N.
\]
\end{remark}

Now suppose that $\X$ is compact and the $\X_i$, $i\in \N_N$, are closed, i.e., compact in $\X$. If in addition the local IFS $\{\X; (\X_i, f_i)\st i \in \N_N\}$ is contractive then the local attractor can be computed as follows. Let $K_0:= \X$ and set
\[
K_n := \cF_\loc (K_{n-1}) = \bigcup_{i\in \N_N} f_i (K_{n-1}\cap \X_i), \quad n\in \N.
\]
Then $\{K_n\st n\in\N_0\}$ is a decreasing nested sequence of compact sets. \textit{If} each $K_n$ is nonempty then by the Cantor Intersection Theorem,
\[
K:= \bigcap_{n\in \N_0} K_n \neq \emptyset.
\]
Using \cite[Proposition 3 (vii)]{lesniak}, we see that
\[
K = \lim_{n\to\infty} K_n,
\]
where the limit is taken with respect to the Hausdorff metric on $\bH$. This implies that
\[
K = \lim_{n\to\infty} K_n = \lim_{n\to\infty} \bigcup_{i\in \N_N} f_i (K_{n-1}\cap \X_i) = \bigcup_{i\in \N_N} f_i (K\cap \X_i) = \cF_\loc (K). 
\]
Thus, $K = A_\loc$. A condition which guarantees that each $K_n$ is nonempty is that $f_i(\X_i) \subset\X_i$, $i\in \N_N$. (See also \cite{barnhurd}.)

In the above setting, one can derive a relation between the local attractor $A_\loc$ of a contractive local IFS $\{\X; (\X_i, f_i)\st i \in \N_N\}$ and the (global) attractor $A$ of the associated (global) IFS $\{\X; f_i\st i \in \N_N\}$. To this end, let the sequence $\{\K_n\st n\in \N_0\}$ be defined as above. The unique attractor $A$ of the IFS $\cF:= \{\X; f_i\st i \in \N_N\}$ is obtained as the fixed point of the set-valued map $\cF: \bH\to \bH$, 
\be\label{setvalued}
\cF (B) = \bigcup_{i\in \N_N} f_i (B),
\ee
where $B\in \bH$. If the IFS $\cF$ is contractive, then the set-valued mapping \eqref{setvalued} is contractive on $\bH$ and its unique fixed point is obtained as the limit of the sequence of sets $\{A_n\st n\in \N_0\}$ with $A_0 := \X$ and 
\[
A_n := \cF(A_{n-1}), \quad n\in \N.
\]
Note that $K_0 = A_0 = \X$ and, assuming that $K_{n-1}\subseteq A_{n-1}$, $n\in\N$, it follows by induction that
\begin{align*}
K_n & = \bigcup_{i\in \N_N} f_i (K_{n-1}\cap X_i) \subseteq \bigcup_{i\in \N_N} f_i (K_{n-1}) \subseteq \bigcup_{i\in \N_N} f_i (A_{n-1}) = A_n.
\end{align*}
Hence, upon taking the limit with respect to the Hausdorff metric as $n\to\infty$, we obtain $A_\loc \subseteq A$. This proves the next result.

\begin{proposition}
Let $\X$ be a compact metric space and let $\X_i$, $i\in \N_N$, be closed, i.e., compact in $\X$. Suppose that the local IFS $\cF_\loc := \{\X; (\X_i, f_i)\st i \in \N_N\}$ and the IFS $\cF:=\{\X; f_i\st i \in \N_N\}$ are both contractive. Then the local attractor $A_\loc$ of $\cF_\loc$ is a subset of the attractor $A$ of $\cF$. 
\end{proposition}

Contractive local IFSs are point-fibered provided $\X$ is compact and the subsets $\X_i$, $i\in \N_N$, are closed. To show this, define the code space of a local IFS by $\Omega:= \prod_{n\in\N}\N_N$ and endowed it with the product topology $\mathfrak{T}$. It is known that $\Omega$ is metrizable and that $\mathfrak{T}$ is induced by the metric $d_F: \Omega\times\Omega\to \R$,
\[
d_F(\sigma,\tau) := \sum_{n\in \N} \frac{|\sigma_n - \tau_n|}{(N+1)^n},
\]
where $\sigma = (\sigma_1\ldots\sigma_n\ldots)$ and $\tau = (\tau_1\ldots\tau_n\ldots)$. (As a reference, see for instance \cite{Engel}, Theorem 4.2.2.) The elements of $\Omega$ are called codes.

Define a set-valued mapping $\gamma :\Omega \to \K(\X)$, where $\K(\X)$ denotes the hyperspace of all compact subsets of $X$, by
\[
\gamma (\sigma) := \bigcap_{n=1}^\infty f_{\sigma_1}\circ \cdots \circ f_{\sigma_n} (\X),
\]
where $\sigma = (\sigma_1\ldots\sigma_n\ldots)$. Then $\gamma (\sigma)$ is point-fibered, i.e., a singleton. Moreover, in this case, the local attractor $A$ equals $\gamma(\Omega)$. (For details about point-fibered IFSs and attractors, we refer the interested reader to \cite{K}, Chapters 3--5.) 

\begin{example}
Let $\X := [0,1]\times [0,1]$ and suppose that $0 < x_2 < x_1 < 1$ and $0 < y_2 < y_1 < 1$. Define
\[
\X_1 := [0,x_1]\times [0,y_1]\qquad\text{and}\qquad \X_2 := [x_2,1]\times [y_2,1].
\]
Furthermore, let $f_i:\X_i \to \X$, $i=1,2$, be given by
\[
f_1(x,y) := (s_1 x, s_1 y)\quad\text{and}\quad f_2(x,y) := (s_2 x + (1-s_2) x_2, s_2 y + (1-s_2)y_2),
\]
respectively, where $s_1,s_2\in [0,1)$.

The (global) IFS $\{\X; f_1, f_2\}$ has the line segment $A = \{(x, \frac{y_2}{x_2}\, x)\st 0\leq x \leq x_2\}$ as its unique attractor. The local attractor of the local IFS $\{\X; (\X_1, f_1), (\X_2, f_2)\}$ is given by $A_\loc = \{(0,0)\}\cup\{(x_2,y_2)\}$, the union of the fixed point for $f_1$ and $f_2$, respectively.
\end{example}
\section{Local Fractal Functions}\label{sec4}
In this section, we introduce bounded local fractal functions as the fixed points of operators acting on the complete metric space of bounded functions. We will see that the graph of a local fractal functions is the local attractor of an associated local IFS and that the set of discontinuities of a bounded local fractal function is at most countably infinite. We follow the exhibition presented in \cite{BHM1}.

To this end, let $X$ be a nonempty connected set and $\{X_i \st i \in\N_N\}$ a family of nonempty connected subsets of $X$. Suppose $\{u_i : X_i\to X \st i \in \N_N\}$ is a family of bijective mappings with the property that
\begin{enumerate}
\item[(P)] $\{u_i(X_i)\st i \in\N_N\}$ forms a (set-theoretic) partition of $X$: $X = \bigcup_{i=1}^N u_i(X_i)$ and $u_i(X_i)\cap u_j(X_j) = \emptyset$, for all $i\neq j\in \N_N$.
\end{enumerate}
\noindent
Now suppose that $(\Y,d_\Y)$ is a complete metric space with metric $d_\Y$. A mapping $f:X\to \Y$ is called \textbf{bounded} (with respect to the metric $d_\Y$) if there exists an $M> 0$ so that for all $x_1, x_2\in X$, $d_\Y(f(x_1),f(x_2)) < M$.

Denote by $B(X, \Y)$ the set
\[
B(X, \Y) := \{f : X\to \Y \st \text{$f$ is bounded}\}.
\]
Endowed with the metric 
\[
d(f,g): = \displaystyle{\sup_{x\in X}} \,d_\Y(f(x), g(x)),
\] 
$(B(X, \Y), d)$ becomes a complete metric space. In a similar fashion, we define $B(X_i, \Y)$, $i \in\N_N$.

Under the usual addition and scalar multiplication of functions, the spaces $B(X_i,\Y)$ and $B(X,\Y)$ become metric linear spaces \cite{Rol}. Recall that a \textbf{metric linear space} is a vector space endowed with a metric under which the operations of vector addition and scalar multiplication become continuous.

For $i \in \N_N$, let $v_i: X_i\times \Y \to \Y$ be a mapping that is uniformly contractive in the second variable, i.e., there exists an $\ell\in [0,1)$ so that for all $y_1, y_2\in \Y$
\be\label{scon}
d_\Y (v_i(x, y_1), v_i(x, y_2)) \leq \ell\, d_\Y (y_1, y_2), \quad\forall x\in X.
\ee
Define a \textbf{Read-Bajactarevi\'c (RB) operator} $\Phi: B(X,\Y)\to \Y^{X}$ by
\be\label{RB}
\Phi f (x) := \sum_{i=1}^N v_i (u_i^{-1} (x), f_i\circ u_i^{-1} (x))\,\chi_{u_i(X_i)}(x), 
\ee
where $f_i := f\vert_{X_i}$ and
$$
\chi_M (x) := \begin{cases} 1, & x\in M\\ 0, & x\notin M\end{cases},
$$
denotes the characteristic function of a set $M$. Note that $\Phi$ is well-defined, and since $f$ is bounded and each $v_i$ contractive in its second variable, $\Phi f\in B(X,\Y)$.

Moreover, by \eqref{scon}, we obtain for all $f,g\in B(X, \Y)$ the following inequality:
\begin{align}\label{estim}
d(\Phi f, \Phi g) & = \sup_{x\in X} d_\Y (\Phi f (x), \Phi g (x))\nonumber\\
& = \sup_{x\in X} d_\Y (v(u_i^{-1} (x), f_i(u_i^{-1} (x))), v(u_i^{-1} (x), g_i(u_i^{-1} (x))))\nonumber\\
& \leq \ell\sup_{x\in X} d_\Y (f_i\circ u_i^{-1} (x), g_i \circ u_i^{-1} (x)) \leq \ell\, d_\Y(f,g).
\end{align}
To simplify notation, we had set $v(x,y):= \sum_{i=1}^N v_i (x, y)\,\chi_{X_i}(x)$ in the above equation. In other words, $\Phi$ is a contraction on the complete metric space $B(X,\Y)$ and, by the Banach Fixed Point Theorem, has therefore a unique fixed point $\ff$ in $B(X,\Y)$. This unique fixed point will be called a \textbf{local fractal function}  $\ff = \ff_\Phi$ (generated by $\Phi$).

Next, we would like to consider a special choice of mappings $v_i$. To this end, we require the concept of an $F$-space. We recall that a metric $d:\Y\times\Y\to \R$ is called \textbf{complete} if every Cauchy sequence in $\Y$ converges with respect to $d$ to a point of $\Y$, and \textbf{translation-invariant} if $d(x+a,y+a) = d(x,y)$, for all $x,y,a\in \Y$.

\begin{definition}
A topological vector space $\Y$ is called an \textbf{$\boldsymbol{F}$-space} \cite{Rol} if its topology is induced by a complete translation-invariant metric $d$.
\end{definition}

Now suppose that $\Y$ is an $F$-space. Denote its metric by $d_\Y$. We define mappings $v_i:X_i\times\Y\to \Y$ by
\be\label{specialv}
v_i (x,y) := \lambda_i (x) + S_i (x) \,y,\quad i \in \N_N,
\ee
where $\lambda_i \in B(X_i,\Y)$ and $S_i : X_i\to \R$ is a function.

If in addition we require that the metric $d_\Y$ is homogeneous, that is,
\[
d_\Y(\alpha y_1, \alpha y_2) = |\alpha| d_\Y(y_1,y_2), \quad \forall \alpha\in\R\;\forall y_1.y_2\in \Y,
\]
then $v_i$ given by \eqref{specialv} satisfies condition \eqref{scon} provided that the functions $S_i$ are bounded on $X_i$ with bounds in $[0,1)$. For then
\begin{align*}
d_\Y (\lambda_i (x) + S_i (x) \,y_1,\lambda_i (x) + S_i (x) \,y_2) &= d_\Y(S_i (x) \,y_1,S_i (x) \,y_2) \\
& = |S_i(x)| d_\Y (y_1, y_2)\\
& \leq \|S_i\|_{\infty,X_i}\, d_\Y (y_1, y_2)\\
& \leq s\,d_\Y (y_1, y_2).
\end{align*}
Here, we denoted the supremum norm with respect to $X_i$ by $\|\bullet\|_{\infty, X_i}$, and set $s := \max\{\|S_i\|_{\infty,X_i}\st$ $i\in \N_N\}$.

Thus, for a {\em fixed} set of functions $\{\lambda_1, \ldots, \lambda_N\}$ and $\{S_1, \ldots, S_N\}$, the associated RB operator \eqref{RB} has now the form
\[
\Phi f = \sum_{i=1}^N \lambda_i\circ u_i^{-1} \,\chi_{u_i(X_i)} + \sum_{i=1}^N (S_i\circ u_i^{-1})\cdot (f_i\circ u_i^{-1})\,\chi_{u_i(X_i)},
\]
or, equivalently,
\[
\Phi f_i\circ u_i = \lambda_i + S_i\cdot f_i, \quad \text{on $X_i$, $\forall\;i\in\N_N$,}
\]
with $f_i = f\vert_{X_i}$.

\begin{theorem}
Let $\Y$ be an $F$-space with homogeneous metric $d_\Y$. Let $X$ be a nonempty connected set and $\{X_i \st i \in\N_N\}$ a collection of nonempty connected subsets of $X$. Suppose that $\{u_i : X_i\to X \st i \in \N_N\}$ is a family of bijective mappings satisfying property $\mathrm{(P)}$.

Let $\blambda := (\lambda_1, \ldots, \lambda_N)\in \underset{i=1}{\overset{N}{\times}} B(X_i,\Y)$ and $\bS := (S_1, \ldots, S_N)\in \underset{i=1}{\overset{N}{\times}} B (X_i,\R)$. Define a mapping $\Phi: \left(\underset{i=1}{\overset{N}{\times}} B(X_i,\Y)\right)\times \left(\underset{i=1}{\overset{N}{\times}} B (X_i,\R)\right) \times B(X,\Y)\to B(X,\Y)$ by
\be\label{eq3.4}
\Phi(\blambda)(\bS) f = \sum_{i=1}^N \lambda_i\circ u_i^{-1} \,\chi_{u_i(X_i)} + \sum_{i=1}^N (S_i\circ u_i^{-1})\cdot (f_i\circ u_i^{-1})\,\chi_{u_i(X_i)}.
\ee
If $\max\{\|S_i\|_{\infty,X_i}\st i\in \N_N\} < 1$ then the operator $\Phi(\blambda)(\bS)$ is contractive on the complete metric space $B(X, \Y)$ and its unique fixed point $\ff$ satisfies the self-referential equation
\be\label{3.4}
\ff = \sum_{i=1}^N \lambda_i\circ u_i^{-1} \,\chi_{u_i(X_i)} + \sum_{i=1}^N (S_i\circ u_i^{-1})\cdot (\ff_i\circ u_i^{-1})\,\chi_{u_i(X_i)},
\ee
or, equivalently
\be
\ff\circ u_i = \lambda_i + S_i\cdot \ff_i, \quad \text{on $X_i$, $\forall\;i\in\N_N$,}
\ee
where $\ff_i = \ff\vert_{X_i}$.
\end{theorem}
\begin{proof}
The statements follow directly from the considerations preceding the theorem.
\end{proof}

The fixed point $\ff$ in \eqref{3.4} is called a \textbf{bounded local fractal function} or, for short, \textbf{local fractal function}.

\begin{remark}
Note that the local fractal function $\ff$ generated by the operator $\Phi$ defined by \eqref{eq3.4} does not only depend on the family of subsets $\{X_i \st i \in \N_N\}$ but also on the two $N$-tuples of bounded functions $\blambda\in \underset{i=1}{\overset{N}{\times}} B(X_i,\Y)$ and $\bS\in \underset{i=1}{\overset{N}{\times}} B (X_i,\R)$. The fixed point $\ff$ should therefore be written more precisely as $\ff (\blambda)(\bS)$. However, for the sake of notational simplicity, we usually suppress this dependence for both $\ff$ and $\Phi$.
\end{remark}

\begin{example}
Suppose $X:= [0,1]$ and $\Y:=\R$. In Figure \ref{fig:randfracfun}, we display the graph of a randomly generated local fractal function where the $\lambda_i$'s and the $S_i$'s were chosen to have random constant values. 
\begin{figure}[h!]
  \centerline{\includegraphics[width=0.5\textwidth]{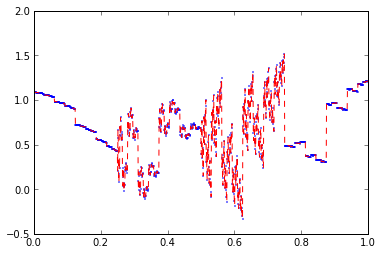}}
\caption{A randomly generated local fractal function\label{fig:randfracfun}}
\end{figure}
\end{example}

The following result found in \cite{GHM} and, in more general form, in \cite{M97} is the extension to the setting of local fractal functions.
\begin{theorem}\label{thm3.3}
The mapping $\blambda \mapsto \ff(\blambda)$ defines a linear isomorphism from $\underset{i=1}{\overset{N}{\times}} B(X_i,\Y)$ to $B(X,\Y)$.
\end{theorem}
\begin{proof}
Let $\alpha, \beta \in\R$ and let $\blambda, \bmu\in \underset{i=1}{\overset{N}{\times}} B(X_i,\Y)$. Injectivity follows immediately from the fixed point equation \eqref{3.4} and the uniqueness of the fixed point: $\blambda = \bmu$ $\Longleftrightarrow$ $\ff(\blambda) = \ff(\bmu)$, .

Linearity follows from \eqref{3.4}, the uniqueness of the fixed point and injectivity: 
\begin{align*}
\ff(\alpha\blambda + \beta \bmu) & = \sum_{i=1}^N (\alpha\lambda_i + \beta \mu_i) \circ u_i^{-1} \,\chi_{u_i(X_i)}\\
& \qquad  + \sum_{i=1}^N (S_i\circ u_i^{-1})\cdot (f_i^*(\alpha\blambda + \beta \bmu)\circ u_i^{-1})\,\chi_{u_i(X_i)}
\end{align*}
and
\begin{align*}
\alpha \ff(\blambda) + \beta \ff(\bmu) & = \sum_{i=1}^N (\alpha\lambda_i + \beta \mu_i) \circ u_i^{-1} \,\chi_{u_i(X_i)}\\
& \qquad  + \sum_{i=1}^N (S_i\circ u_i^{-1})\cdot (\alpha f_i^*(\blambda) + \beta f_i^*(\bmu))\circ u_i^{-1})\,\chi_{u_i(X_i)}.
\end{align*}
Hence, $\ff(\alpha\blambda + \beta \bmu) = \alpha \ff(\blambda) + \beta \ff(\bmu)$.

For surjectivity, we define $\lambda_i := \ff\circ u_i - S_i \cdot \ff$, $i\in \N_N$. Since $\ff\in B(X,\Y)$, we have $\blambda\in \underset{i=1}{\overset{N}{\times}} B(X_i,\Y)$. Thus, $\ff(\blambda) = \ff$.
\end{proof}

The next results gives information about the set of discontinuities of a bounded local fractal function $\ff$. The proof can be found in \cite{BHM1}.

\begin{theorem}\label{discont}
Let $\Phi$ be given as in \eqref{eq3.4}. Assume that for all $i\in \N_N$ the $u_i$ are contractive and the $\lambda_i$ are continuous on $\overline{X_i}$. Further assume that 
\[
\max\left\{\|S_i\|_{\infty,X_i}\st i\in\N_N\right\} < 1,
\]
 and that the fixed point $\ff$ is bounded everywhere. Then the set of discontinuities of $\ff$ is at most countably infinite.
\end{theorem}

Next, we exhibit the relation between the graph $G$ of the fixed point $\ff$ of the operator $\Phi$ given by \eqref{eq3.4} and the local attractor of an associated contractive local IFS. To this end, we need to require that $\X$ is a closed subset of a complete metric space, hence complete itself. Consider the complete metric space $\X\times\Y$ and define mappings $w_i:\X_i\times\Y\to \X\times\Y$ by
\[
w_i (x, y) := (u_i (x), v_i (x,y)), \quad i\in \N_N.
\]
Assume that the mappings $v_i: \X_i\times \Y\to \Y$ in addition to being uniformly contractive in the second variable are also uniformly Lipschitz continuous in the first variable, i.e., that there exists a constant $L > 0$ so that for all $y\in \Y$,
\[
d_\Y(v_i(x_1, y),v_i(x_2, y)) \leq L \, d_\X (x_1,x_2), \quad\forall x_1, x_2\in \X_i,\quad\forall i\in \N_N.
\]
Denote by $a:= \max\{a_i\st i\in \N_N\}$ the largest of the Lipschitz constants of the mappings $u_i:\X_i\to \X$ and let $\theta := \frac{1-a}{2L}$. It is straight-forward to show that the mapping $d_\theta : (\X\times\Y)\times(\X\times\Y) \to \R$ given by
\[
d_\theta := d_\X + \theta\,d_\Y
\]
is a metric for $\X\times\Y$ compatible with the product topology on $\X\times\Y$.

\begin{theorem}
The family $\cW_\loc := \{\X\times\Y; (\X_i\times\Y, w_i)\st i\in \N_N\}$ is a contractive local IFS in the metric $d_\theta$ and the graph $G(\ff)$ of the local fractal function $\ff$ associated with the operator $\Phi$ given by \eqref{eq3.4} is an attractor of $\cW_\loc$. Moreover, 
\be\label{GW}
G(\Phi \ff) = \cW_\loc (G(\ff)),
\ee
where $\cW_\loc$ denotes the set-valued operator \eqref{hutchop} associated with the local IFS $\cW_\loc$.
\end{theorem}

\begin{proof}
We first show that $\{\X\times\Y; (X_i\times\Y, w_i)\st i\in \N_N\}$ is a contractive local IFS. For this purpose, let $(x_1,y_1), (x_2,y_2)\in \X_i\times\Y$, $i\in \N_N$, and note that
\begin{align*}
d_\theta (w_i(x_1,y_1), w_i(x_2,y_2)) & = d_\X (u_i (x_1), u_i(x_2)) + \theta d_\Y(v_i (x_1,y_1), v_i (x_2,y_2)) \\
& \leq a\, d_\X(x_1, x_2) + \theta d_\Y(v_i (x_1,y_1), v_i (x_2,y_1))\\ 
& \qquad + \theta d_\Y(v_i (x_2,y_1), v_i (x_2,y_2))\\
& \leq (a + \theta L) d_\X(x_1, x_2) + \theta\,s \,d_\Y(y_1,y_2) \\
& \leq q\,d_\theta ((x_1,y_1), (x_2,y_2)).
\end{align*}
Here we used \eqref{scon} and set $q:= \max\{a + \theta L, s\} < 1$. 

The graph $G(\ff)$ of $\ff$ is an attractor for the contractive local IFS $\cW_\loc$, for
\begin{align*}
\cW_\loc (G(\ff)) & = \bigcup_{i=1}^N w_i (G(\ff)\cap \X_i) = \bigcup_{i=1}^N w_i (\{(x, \ff(x)\st x\in \X_i\}\\
& = \bigcup_{i=1}^N \{(u_i (x), v_i(x, \ff(x)))\st x\in \X_i\} = \bigcup_{i=1}^N \{(u_i(x), \ff(u_i(x)))\st x\in \X_i\}\\ 
& = \bigcup_{i=1}^N \{(x, \ff(x)) \st x\in u_i(\X_i)\} = G(\ff).
\end{align*}
That \eqref{GW} holds follows from the above computation and the fixed point equation for $\ff$ written in the form
\[
\ff\circ u_i (x) = v_i (x, \ff (x)), \quad x\in \X_i, \quad i\in \N_N.\qedhere
\]
\end{proof}
\section{Tensor Products of Local Fractal Functions}\label{sec5}
In this section, we define the tensor product of local fractal functions thus extending the previous construction to higher dimensions.

For this purpose, we follow the notation and of the previous section, and assume that $X$ and $\oX$ are nonempty connected sets, and $\{X_i\st i\in \N_N\}$ and $\{\oXi\st i\in \N_N\}$ are families of nonempty connected subsets of $X$ and $\oX$, respectively. Analogously, we define finite families of bijections $\{u_i: X_i\to X\st i\in \N_N\}$ and $\{\ou_i: \oXi\to X\st i\in \N_N\}$ requiring both to satisfy condition (P).

Furthermore, we assume that $(\Y, \|\bullet\|_\Y)$ is a \textbf{Banach algebra}, i.e., a Banach space that is also an associate algebra for which multiplication is continuous: 
$$
\|y_1y_2\|_\Y \leq \|y_1\|_\Y\,\|y_2\|_\Y, \quad\forall\,y_1,y_2\in \Y. 
$$
Let $f\in B(X,\Y)$ and $\of\in B(\oX,\Y)$. The tensor product of $f$ with $\of$, written $f\otimes\of: X\times\oX\to \Y$, with values in $\Y$ is defined by
\[
(f\otimes\of) (x,\ox) := f(x) \of(\ox),\quad\forall\,(x,\ox)\in X\times\oX.
\]
As $f$ and $\of$ are bounded, the inequality
\[
\|(f\otimes\of)(x,\ox)\|_{\Y} = \|f(x)\of(\ox\|_{\Y} \leq  \|f(x)\|_\Y \, \|\of(\ox)\|_\Y, 
\]
implies that $f\otimes\of$ is bounded. Under the usual addition and scalar multiplication of functions, the set
\[
B(X\times\oX, \Y) := \{f\otimes\of : X\times \oX\to \Y \st \text{$f\otimes\of$ is bounded}\}
\]
becomes a complete metric space when endowed with the metric
\[
d(f\otimes\of, g\otimes\og) := \sup_{x\in X} \|f(x) - g(x)\|_\Y + \sup_{\ox\in\oX} \|\of(\ox) - \og(\ox)\|_\Y.
\]
Now let $\Phi: B(X,\Y)\to B(X,\Y)$ and $\oPhi: B(\oX,\Y)\to B(\oX,\Y)$ be contractive RB-operators of the form \eqref{RB}. We define the tensor product of $\Phi$ with $\oPhi$ to be the RB-operator $\Phi\otimes\oPhi: B(X\times\oX, \Y)\to B(X\times\oX, \Y)$ given by
\[
(\Phi\otimes\oPhi)(f\otimes\of) := (\Phi f)\otimes (\oPhi \of).
\]
It follows that $\Phi\otimes\oPhi$ maps bounded functions to bounded functions. Furthermore, $\Phi\otimes\oPhi$ is contractive on the complete metric space $(B(X\times\oX, \Y),d)$. To see this, note that
\begin{align*}
\sup_{x\in X}\|(\Phi f)(x) &- (\Phi g)(x)\|_\Y + \sup_{\ox\in \oX}\|(\Phi \of)(\ox) - (\Phi \og)(\ox)\|_\Y \\
& \leq \ell \sup_{x\in X}\|f(x) - g(x)\|_\Y + \overline{\ell} \sup_{\ox\in \oX} \|\of(\ox) - \og(\ox)\|_\Y\\
& \leq \max\{\ell, \overline{\ell}\}\, d(f\otimes\of, g\otimes\og),
\end{align*}
where we used \eqref{estim} and denoted the uniform contractivity constant of $\oPhi$ by $\overline{\ell}$.

The unique fixed point of the RB-operator $\Phi\otimes\oPhi$ will be called a \textbf{tensor product local fractal function} and its graph a \textbf{tensor product local fractal surface}.
\section{Lebesgue Spaces $L^p(\R)$}\label{sec6}
We may construct local fractal functions on spaces other than $B(X,\Y)$. (See also \cite{BHM1}.) In this section, we derive conditions under which local fractal functions are elements of the Lebesgue spaces $L^p$ for $p>0$. To this end, we assume again that the functions $v_i$ are given by \eqref{specialv} and that $\X := [0,1]$ and $\Y := \R$. We consider the metric on $\R$ and $\X=[0,1]$ as being induced by the $L^1$-norm. Note that endowed with this norm $B(\X,\R)$ becomes a Banach space.

Recall that the Lebesgue spaces $L^p [0,1]$, $1\leq p\leq \infty$, are obtained as the completion of the space $C[0,1]$ of real-valued continuous functions on $[0,1]$ with respect to the $L^p$-norm
\[
\|f\|_{L^p} := \left(\int_{[0,1]} |f(x)|^p \,dx\right)^{1/p}.
\]
For $0 < p <1$, the spaces $L^p(\R)$ are defined as above but instead of a norm, a metric is used to obtain completeness. More precisely, define
\[
d_p (f,g) := \|f - g\|_{L^p}^p,
\]
where $\|\bullet\|_{L^p}$ is the norm introduced above. Then $(L^p(\R), d_p)$ is an $F$-space. (Note that the inequality $(a+b)^p \leq a^p + b^p$ holds for all $a,b\geq 0$.) For more details, we refer to \cite{rudin}.

We have the following result for RB-operators defined on the Lebesgue spaces $L^p[0,1]$, $0 < p \leq \infty$. The case $p\in [1, \infty]$ was already considered in \cite{BHM1}, but for the sake of completeness we reproduce the proof.

\begin{theorem}\label{thm7}
Suppose that $\{X_i \st i \in \N_N\}$ is a family of half-open intervals of $[0,1]$. Further suppose that $\{x_0 := 0 < x_1 < \cdots < x_N := 1\}$ is a partition of $[0,1]$ and that $\{u_i \st i \in\N_N\}$ is a family of affine mappings from $X_i$ onto $[x_{i-1}, x_i)$, $i = 1, \ldots, N-1$, and from $X_N^+ := X_N\cup u_N^{-1}(1-)$ onto $[x_{N-1},x_N]$, where $u_N$ maps $X_N$ onto $[x_{N-1}, x_N)$. 

The operator $\Phi: L^p [0,1]\to \R^{[0,1]}$, $p\in (0,\infty]$, defined by
\be\label{Phi}
\Phi g := \sum_{i=1}^N (\lambda_i \circ u_i^{-1})\,\chi_{u_i(X_i)} + \sum_{i=1}^N (S_i\circ u_i^{-1})\cdot (g_i\circ u_i^{-1})\,\chi_{u_i(X_i)},
\ee
where $g_i = g\vert_{X_i}$, $\lambda_i\in L^p (X_i, [0,1])$ and $S_i\in L^\infty (X_i, \R)$, $i \in\N_N$, maps $L^p [0,1]$ into itself. Moreover, if 
\be\label{condition}
\begin{cases}
\displaystyle{\sum_{i=1}^N}\, a_i \,\|S_i\|_{\infty, X_i}^p < 1, & p\in (0,1);\\ \\
\left(\displaystyle{\sum_{i=1}^N}\, a_i \,\|S_i\|_{\infty, X_i}^p\right)^{1/p} < 1, & p\in[1,\infty);\\ \\
\max\left\{\|S_i\|_{\infty,X_i}\st i\in\N_N\right\} < 1, & p = \infty,
\end{cases}
\ee
where $a_i$ denotes the Lipschitz constant of $u_i$, then $\Phi$ is contractive on $L^p [0,1]$ and its unique fixed point $\ff$ is an element of $L^p [0,1]$.
\end{theorem}

\begin{proof}
Note that under the hypotheses on the functions $\lambda_i$ and $S_i$ as well as the mappings $u_i$, $\Phi f$ is well-defined and an element of $L^p[0,1]$. It remains to be shown that under conditions \eqref{condition}, $\Phi$ is contractive on $L^p[0,1]
$. 

We start with $1\leq p<\infty$. If $g,h \in L^p [0,1]$ then
\begin{align*}
\|\Phi g - \Phi h\|^{L^p}_{p} & = \int\limits_{[0,1]} |\Phi g (x) - \Phi h (x)|^p dx\\
& = \int\limits_{[0,1]} \left|\sum_{i=1}^{N} (S_i\circ u_i^{-1})(x) [(g_i\circ u_i^{-1})(x) - (h_i\circ u_i^{-1})(x)]\,\chi_{u_i(X_i)}(x)\right|^p\, dx\\
& = \sum_{i=1}^{N}\,\int\limits_{[x_{i-1},x_i]}\left| (S_i\circ u_i^{-1})(x) [(g_i\circ u_i^{-1})(x) - (h_i\circ u_i^{-1})(x)]\right|^p\,dx\\
& = \sum_{i=1}^{N}\,a_i\,\int\limits_{X_i} \left| S_i (x) [g_i(x)- h_i(x)]\right|^p\,dx\\
&  \leq \sum_{i=1}^{N}\,a_i\,\|S_i\|^p_{\infty, X_i}\,\int\limits_{X_i} \left| g_i(x) - h_i(x)\right|^p\,dx = \sum_{i=1}^{N}\,a_i\,\|S_i\|^p_{\infty, X_i}\,\|g_i - h_i\|^p_{{L^p},X_i}\\
& = \sum_{i=1}^{N}\,a_i\,\|S_i\|^p_{\infty, X_i}\,\|g_i - h_i\|^p_{{L^p}} \leq \left(\sum_{i=1}^{N}\,a_i\,\|S_i\|^p_{\infty, X_i}\right) \|g - h\|^p_{{L^p}}.
\end{align*}
The case $0<p<1$ now follows in similar fashion. We again have after substitution and rearrangement 
\begin{align*}
d_p(\Phi g,\Phi h) & = \sum_{i=1}^{N}\,a_i\,\int\limits_{X_i} \left| S_i (x) [g_i(x)- h_i(x)]\right|^p\,dx\\
& = \sum_{i=1}^{N}\,a_i\,\|S_i\|^p_{\infty, X_i}\,\|g_i - h_i\|^p_{{L^p}} \leq \left(\sum_{i=1}^{N}\,a_i\,\|S_i\|^p_{\infty, X_i}\right) \|g - h\|^p_{{L^p}}\\
& = \left(\sum_{i=1}^{N}\,a_i\,\|S_i\|^p_{\infty, X_i}\right) d_p(g, h).
\end{align*}
Now let $p= \infty$. Then
\begin{align*}
\|\Phi g - \Phi h\|_{\infty} & = \left\|\sum_{i=1}^{N} (S_i\circ u_i^{-1})(x) [(g_i\circ u_i^{-1})(x) - (h_i\circ u_i^{-1})(x)]\,\chi_{u_i(X_i)}(x)\right\|_\infty\\
& \leq \max_{i\in\N_N}\,\left\| (S_i\circ u_i^{-1})(x) [(g_i\circ u_i^{-1})(x) - (h_i\circ u_i^{-1})(x)]\right\|_{\infty,X_i}\\
& \leq \max_{i\in\N_N}\|S_i\|_{\infty,X_i} \left\|g_i - h_i]\right\|_{\infty,X_i} = \max_{i\in\N_N}\|S_i\|_{\infty,X_i} \left\|g_i - h_i]\right\|_{\infty}\\
&  \leq \left(\max_{{i\in\N_N}}\,\|S_i\|_{\infty,X_i}\right) \left\|g - h]\right\|_{\infty}
\end{align*}
These calculations prove the claims.
\end{proof}

\begin{remark}
The proof of the theorem shows that the conclusions also hold under the assumption that the family of mappings $\{u_i: X_i\to \X\st i\in \N_N\}$ is generated by the following functions.
\begin{enumerate}
\item[$\mathrm{(i)}$] Each $u_i$ is a bounded diffeomorphism of class $C^k$, $k\in \N\cup\{\infty\}$, from $X_i$ to $[x_{i-1}, x_i)$ (obvious modification for $i = N$). In this case, the $a_i$'s are given by $a_i = \sup\{\left\vert \frac{du_i}{dx} (x)\right\vert\st x$ $\in X_i\}$, $i\in\N_N$.
\item[$\mathrm{(ii)}$] Each $u_i$ is a bounded invertible function in $C^\omega$, the class of real-analytic functions from $X_i$ to $[x_{i-1}, x_i)$ and its inverse is also in $C^\omega$. (Obvious modification for $i = N$.) The $a_i$'s are given as above in item $\mathrm{(i)}$.
\end{enumerate}
\end{remark}
\section{Smoothness spaces $C^n$ and H\"older Spaces $\dot{C}^s$}\label{sec7}
Our next objective is to derive conditions on the partition $\{X_i\st i\in \N_N\}$ of $\X:= [0,1]$ and the function tuples $\blambda$ and $\bS$ so that we obtain a continuous or even differentiable local fractal function $\ff:[0,1]\to\R$. To this end, consider the complete metric linear space $C := C^0 (\X) := \{f: [0,1]\to \R\st \text{$f$ continuous}\}$ endowed with the supremum norm $\|\bullet\|_\infty$.

\subsection{Binary partition of $\X$}
We introduce the following subsets of $\X=[0,1]$ which play an important role in fractal-based numerical analysis as they give discretizations for efficient computations. For more details, we refer to \cite{BHM1} and partly to \cite{BHM}. 

Assume that $N\in 2\N$ and let
\be\label{subsets}
\X_{2j-1} := \X_{2j} := \left[\frac{2(j-1)}{N},\frac{2j}{N}\right], \quad j = 1, \ldots, \frac2N.
\ee
Define affine mappings $u_i:\X_i\to [0,1]$ so that
\be\label{uis}
u_i(\X_i) := \left[\frac{i - 1}{N},\frac{i}{N}\right], \quad i= 1, \ldots, N.
\ee
In explicit form, the $u_i$'s are given by
\[
u_{2j-1} (x) = \frac{x}{2} + \frac{j-1}{N} \quad\text{and}\quad u_{2j} (x) = \frac{x}{2} + \frac{j}{N}, \quad x \in \X_{2j-1} = \X_{2j}.
\]
Note that here $u_i(\X_i) \subsetneq \X_i$, $\forall\,i\in\N_N$. Clearly, $\{u_i(\X_i)\st i\in \N_N\}$ is a partition of $[0,1]$. We denote the distinct endpoints of the partitioning intervals $\{u_i(\X_{i})\}$ by $\{x_0 < x_1 <\ldots < x_{N}\}$ where $x_0 = 0$ and $x_N = 1$, and refer to them as \textbf{knot points} or simply as \textbf{knots}.

Furthermore, we assume that we are given interpolation values at the endpoints of the intervals $\X_{2j-1} = \X_{2j}$:
\be\label{I}
\mathscr{I} := \left\{(x_{2j}, y_j)\st j = 0, 1, \ldots, N/2\right\}.
\ee
Let 
$$
C_{\mathscr{I}} := \{f\in C \st f(x_{2j}) = y_j, \,\forall\, j = 0, 1, \ldots, N/2\}.
$$
Then $C_{\mathscr{I}}$ is a closed metric subspace of $C$. We consider an RB operator $\Phi$ of the form \eqref{eq3.4} acting on $C_{\mathscr{I}}$. 

In order for $\Phi$ to map $C_{\mathscr{I}}$ into itself one needs to require that $\lambda_i, S_i \in C(\X_i) := C(\X_i,\R) := \{f: \X_i\to \R\st \text{$f$ continuous}\}$ and that
\be\label{intcon}
y_{j-1} = \Phi f (x_{2(j-1)}) \quad \wedge \quad y_{j} = \Phi f (x_{2j}), \quad j = 1, \ldots, N/2,
\ee
where $x_{2j} := (2j)/N$. Note that the preimages of the knots $x_{2(j-1)}$ and $x_{2j}$ are the endpoints of $\X_{2j-1} = \X_{2j}$. Substituting the expression for $\Phi$ into \eqref{intcon} and collecting terms yields
\be\label{intcons}
\begin{split}
\lambda_{2j-1} (x_{2(j-1)}) + \left(S_{2j-1}(x_{2(j-1)}) - 1\right) y_{j-1} & = 0,\\
\lambda_{2j} (x_{2j}) + \left(S_{2j}(x_{2j}) - 1\right) y_{j} & = 0,
\end{split}
\ee
for all $j=1, \ldots, N/2$.

To ensure continuity of $\Phi f$ across $[0,1]$, the following join-up conditions at the oddly indexed knots need to be imposed. (They are the images of the midpoints of the intervals $\X_{2j-1} = \X_{2j}$.)
\be\label{prejoin}
\Phi f (x_{2j-1}-)  = \Phi f (x_{2j-1}+) , \quad j = 1, \ldots, N/2.
\ee
A simple calculation gives
\be\label{joinup}
\lambda_{2j} (x_{2(j-1)}) + S_{2j} (x_{2(j-1)}) y_{j-1} = \lambda_{2j-1} (x_{2j}) + S_{2j-1} (x_{2j})  y_j,
\ee
for all $j = 1, \ldots, N/2$. In case all functions $\lambda_i$ and $S_i$ are constant, \eqref{joinup} reduces to the condition given in \cite[Example 2]{BHM}. Two tuples of functions $\blambda, \bS \in \underset{i=1}{\overset{N}{\times}} C(\X_i)$ are said to have property (J) if they satisfy \eqref{intcons} and \eqref{joinup}.

We summarize these results in the next theorem.
\begin{theorem}\label{thm8}
Let $\X:= [0,1]$ and let $N\in 2\N$. Suppose that subsets of $\X$ are given by \eqref{subsets} and the associated mappings $u_i$ by \eqref{uis}. Further suppose that $\mathscr{I}$ is as in \eqref{I} and that $\blambda, \bS \in \underset{i=1}{\overset{N}{\times}} C(\X_i)$ have property (J). Then the RB operator $\Phi$ as given in \eqref{eq3.4} maps $C_\mathscr{I}$ into itself and is well-defined. If, in addition, $\max\left\{\|S_i\|_{\infty,\X_i}\st i\in\N_N\right\} < 1$, then $\Phi$ is a contraction and thus possesses a unique fixed point $\ff:[0,1]\to \R$ satisfying $\ff(x_{2j}) = y_j$, $\forall\, j = 0, 1, \ldots, N/2$.
\end{theorem}

We call this unique fixed point a \textbf{continuous local fractal interpolation function}.

\begin{proof}
It remains to be shown that under the condition $\max\left\{\|S_i\|_{\infty,\X_i}\st i\in\N_N\right\} < 1$, $\Phi$ is contractive on $C_\mathscr{I}$. This, however, follows immediately from the case $p=\infty$ in the proof of Theorem \ref{thm7}.
\end{proof}

Theorem \ref{thm8} can be adapted to the setting of H\"older spaces. For this purpose, we introduce the \textbf{homogeneous H\"older space $\dot{C}^s(\Omega)$}, $0 < s < 1$, as the family of all functions $f\in C(\Omega)$, $\Omega\subseteq\R$, for which
\[
|f|_{\dot{C}^s(\Omega)} := \sup_{x\neq x' \in \Omega} \frac{|f(x) - f(x')|}{|x - x'|^s} < \infty.
\]
$|\bullet|_{\dot{C}^s(\Omega)}$ is a homogeneous semi-norm making $\dot{C}^s$ into a complete locally convex topological vector space, i.e., a Fr\'echet space.

\begin{theorem}
Let $\X:= [0,1]$ and let $N\in 2\N$. Assume that subsets of $\X$ are given by \eqref{subsets}, associated mappings $u_i$ by \eqref{uis}, and that $\mathscr{I}$ is as in \eqref{I}. Assume further that $\blambda\in  \underset{i=1}{\overset{N}{\times}} \dot{C}^s(\X_i)$, $\bS\in \underset{i=1}{\overset{N}{\times}} C(\X_i)$, and that have property condition (J). Then the RB operator \eqref{eq3.4} maps $\dot{C}^s := \dot{C}^s (\X)$ into itself and is well defined. Furthermore, if
\[
2^s \max\left\{\|S_i\|_{\infty,\X_i}\st i\in\N_N\right\} < 1
\]
then $\Phi$ is contractive on $\dot{C}^s$ and has a unique fixed point $\ff\in\dot{C}^s$.
\end{theorem}

In case the last conclusion of the above theorem holds, we say that the fixed point $\ff$ is a \textbf{local fractal function of class $\dot{C}^s$}.
\begin{proof}
First we show that $\Phi f\in \dot{C}^s$. For $x,x'\in [0,1]$, note that there exist $i,i'\in \N_N$ so that $x\in u_i(\X_i)$ and $x'\in u_{i'}(\X_{i'})$. Therefore,
\begin{align*}
|\Phi f(x) - \Phi f(x')| & \leq \left\vert \lambda_i(u_i^{-1}(x)) - \lambda_{i'}(u_{i'}^{-1}(x')) \right\vert \\
& \quad + \left\vert(S_i (u_i^{-1}(x))\cdot (f_i(u_i^{-1}(x)) - (S_{i'} (u_{i'}^{-1}(x'))\cdot (f_{i'}(u_{i'}^{-1}(x'))\right\vert\\
& \leq \left\vert \lambda_i(u_i^{-1}(x)) - \lambda_{i'}(u_{i'}^{-1}(x')) \right\vert\\
& \quad + \max\left\{\|S_i\|_{\infty,\X_i}\right\}
\left\vert f_i(u_i^{-1}(x)) - f_{i'}(u_{i'}^{-1}(x'))\right\vert.
\end{align*}
Using the fact that $|x - x'|^s = 2^{-s} |u_i^{-1}(x) - u_{i'}^{-1}(x')|$ and employing the the properties of the supremum, we thus obtain
\begin{align*}
|\Phi f|_{\dot{C}^s} \leq 2^s \left(\sum_{i\in\N_N} |\lambda_i|_{\dot{C}^s(X_i)} + \max\left\{\|S_i\|_{\infty,\X_i}\right\} |f|_{\dot{C}^s}\right) < \infty.
\end{align*}
To establish the contractivity of $\Phi$, note that
\begin{align*}
|(\Phi f - \,&\Phi g)(x) - (\Phi f - \Phi g)(x')| = \\
& |S_i (u_i^{-1}(x))\cdot (f_i - g_i)(u_i^{-1}(x)) - S_{i'} (u_{i'}^{-1}(x'))\cdot (f_{i'} - g_{i'})(u_{i'}^{-1}(x'))|\\
& \leq \max\left\{\|S_i\|_{\infty,\X_i}\right\} |(f_i - g_i)(u_i^{-1}(x)) - (f_{i'} - g_{i'})(u_{i'}^{-1}(x'))|
\end{align*}
As above, using again $|x - x'|^s = 2^{-s} |u_i^{-1}(x) - u_{i'}^{-1}(x')|$ and that $f$ is defined on all of $[0,1]$, this yields
\[
|\Phi f - \Phi g|_{\dot{C}^s} \leq 2^s \max\left\{\|S_i\|_{\infty,\X_i}\right\} |f - g|_{\dot{C}^s}.\qedhere
\]
\end{proof}

Just as in the case of splines, we can impose join-up conditions and choose the function tuples $\blambda$ and $\bS$ so that the RB operator \eqref{eq3.4} maps the space of continuously differentiable functions into itself. More precisely, suppose that $\Omega\subseteq \R$. Let $C^n (\Omega):= C^n(\Omega,\R) := \{f: \Omega\to \R\st D^k f \in C, \,\forall k = 1, \ldots,n\}$, where $D$ denotes the ordinary differential operator. The linear space $C^n(\Omega)$ is a Banach space under the norm
\[
\|f\|_{C^n(\Omega)} := \sum_{k=0}^n \|D^k f\|_{\infty, \Omega}.
\]
We write $C^n$ for $C^n (\X)$, and will delete the $\Omega$ from the norm notation when $\Omega := \X = [0,1]$.

As we require $C^n$-differentiability across $\X =[0,1]$, we impose $C^n$-interpolation values at the endpoints of the intervals $\X_{2j-1} = \X_{2j}$:
\be\label{In}
\mathscr{I}^{(n)} := \left\{(x_{2j}, \by_j^{(n)})\st j = 0, 1, \ldots, N/2\right\},
\ee
where $\by_j^{(n)} := (y_j^{(0)},y_j^{(1)},\ldots,y_j^{(n)})^T\in \R^{n+1}$ is a given interpolation vector. Let 
$$
C^n_{\mathscr{I}^{(n)}} := \{f\in C^n \st D^k f(x_{2j}) = y_j^{(k)}, \,\forall\, k = 0,1,\ldots, n; \,\forall\, j = 0, 1, \ldots, N/2\}.
$$ 
Then $C^n_{\mathscr{I}^{(n)}}$ is a closed metric subspace of $C^n$.

In order for $\Phi$ to map $C^n_{\mathscr{I}^{(n)}}$ into itself, choose $\lambda_i, S_i \in C^n(\X_i)$, $i\in\N_N$, so that
\be\label{diffintcon}
y_{j-1}^{(k)} = D^k\Phi f (x_{2(j-1)}) \quad \wedge \quad y_{j}^{(k)} = D^k \Phi f (x_{2j}),
\ee
for all $k = 0, 1, \ldots, n$ and for all $j = 1, \ldots, N/2$.

At the midpoints of the intervals $\X_{2j-1} = \X_{2j}$, the function tuples $\blambda$ and $\bS$ need to additionally satisfy the $C^n$-join-up conditions
\be\label{diffcon}
D^k \Phi f (x_{2j-1}-)  = D^k\Phi f (x_{2j-1}+) , \quad\,\forall k = 0,1,\ldots, n;\,\forall j = 1, \ldots, N/2.
\ee

\begin{theorem}
Let $\X:= [0,1]$ and let $N\in 2\N$. Assume that subsets of $\X$ are given by \eqref{subsets}, associated mappings $u_i$ by \eqref{uis}, and that $\mathscr{I}^{(n)}$ is as in \eqref{In}. Assume further that $\blambda, \bS\in  \underset{i=1}{\overset{N}{\times}} C^n(\X_i)$, and that they satisfy conditions \eqref{diffintcon} and \eqref{diffcon}. Then the RB operator \eqref{eq3.4} maps $C^n_{\mathscr{I}^{(n)}}$ into itself and is well defined. Furthermore, if
\be\label{123}
2^n \max_{i\in\N_N}\max_{k=0,1,\ldots,n} \left\{\sum_{l=0}^k \binom{n-k+l}{l}\|D^lS_i\|_{\infty,\X_i}\right\} < 1
\ee
then $\Phi$ is contractive on $C^n_{\mathscr{I}^{(n)}}$ and has a unique fixed point $\ff\in\C^n_{\mathscr{I}^{(n)}}$.
\end{theorem}

We refer to this fixed point $\ff$ as a l\textbf{ocal fractal function of class $C^n_{\mathscr{I}^{(n)}}$}.

\begin{proof}
The statements that $\Phi$ is well defined and maps $C^n_{\mathscr{I}^{(n)}}$ into itself is implied by the conditions imposed on $\blambda$ and $\bS$. It remains to be shown that under condition \eqref{123} the RB operator $\Phi$ is contractive. To this end, consider $f,g\in C^n_{\mathscr{I}^{(n)}}$. Then
\begin{align*}
D^k\Phi f (x) - &\, D^k\Phi g(x) = \sum_{i\in\N_N} D^k\left[S_i(u_i^{-1}(x)) \cdot (f_i(u_i^{-1}(x)) - g_i(u_i^{-1}(x)))\right]\chi_{u_i(\X_i)}\\
& = \sum_{i\in\N_N} \sum_{l=0}^k \binom{k}{l}\, 2^k\left[(D^{k-l}(f_i - g_i))(u_i^{-1}(x)) \cdot (D^l S_i)(u_i^{-1}(x)\right]\chi_{u_i(\X_i)},\\
\end{align*}
where we applied the Leibnitz Differentiation Rule. Therefore, 
\begin{align*}
\|D^k \Phi f - D^k \Phi g\|_\infty &\leq 2^k  \sum_{i\in\N_N} \sum_{l=0}^k \binom{k}{l} \|D^l S_i\|_{\infty,\X_i} \|D^{k-l} (f-g)\|_\infty.
\end{align*}
Hence, 
\begin{align*}
\|\Phi f - \Phi g\|_{C^n} & = \sum_{k=0}^n \|D^k \Phi f - D^k \Phi g\|_\infty\\
& \leq 2^n  \sum_{i\in\N_N}\sum_{k=0}^n \sum_{l=0}^k \binom{k}{l} \|D^l S_i\|_{\infty,\X_i} \|D^{k-l} (f-g)\|_\infty\\
& =  2^n  \sum_{i\in\N_N}\sum_{k=0}^n \sum_{l=0}^k \binom{n-k+l}{l} \|D^l S_i\|_{\infty,\X_i} \|D^{n-k} (f-g)\|_\infty
\end{align*}
The last equality is proven directly by computation or mathematical induction. Thus,
\[
\|\Phi f - \Phi g\|_{C^n} \leq \left(2^n  \max_{i\in\N_N}\max_{k=0,1,\ldots,n} \left\{\sum_{l=0}^k \binom{n-k+l}{l}\|D^lS_i\|_{\infty,\X_i}\right\}\right) \|f - g\|_{C^n},
\]
and the statement follows.
\end{proof}

\subsection{Vanishing endpoint conditions for $S_i$}
Here, we consider a more general set-up than in the previous subsection. We assume again that $\X := [0,1]$ and let $\X_i:=[a_i,b_i]$, for $i\in\N_N$, be $N$ different subintervals of positive length. We further assume that $\{0 =: x_0 < x_1 < \ldots < x_{N-1} < x_N := 1\}$ is a partition of $\X$ and that we have chosen an enumeration in such a way that the mappings $u_i:\X_i\to \X$ satisfy
\[
u_i ([a_i, b_i]) := [x_{i-1}, x_i], \quad\forall\,i \in \N_N.
\]
In particular, note that $a_1 = x_0$, $b_N = x_N$, and $u_i(b_i) = x_i = u_{i+1}(a_{i+1})$, for all interior knots $x_1, \ldots, x_{N-1}$. We assume that the $u_i$ are affine functions but that they are not necessarily contractive.

Let
\be\label{Ia}
\mathscr{I} := \left\{(x_{j}, y_j)\st j = 0, 1, \ldots, N\right\}.
\ee
be a given set of interpolation points and let 
\be\label{111}
C_{\mathscr{I}} := \{f\in C \st f(x_{j}) = y_j, \,\forall\, j = 0, 1, \ldots, N\}.
\ee
Our objective in this subsection is to construct a local fractal function that belongs to $C_{\mathscr{I}}$ and which is generated by an RB operator of the form \eqref{eq3.4}. For this purpose, we need to impose continuity conditions at the interpolation points. More precisely, we require that for an $f\in C_{\mathscr{I}}$,
\be\label{condS}
\begin{split}
\Phi f(x_0)& = y_0, \quad \Phi f(x_N) = y_N,\\
\Phi f (x_i-) = y_i = &\,\Phi f (x_i+),\quad i = 1, \ldots, N-1.
\end{split}
\ee
Substituting the expression for $\Phi$ into these equations and simplifying yields
\begin{gather*}
\lambda_1 (x_0) + S_1(x_0) y_0 = y_0, \quad \lambda_N (x_N) + S_N(x_N) y_N = y_N\\
\lambda_i(b_i) + S_i(b_i) f(b_i) = y_i = \lambda_{i+1} (a_{i+1}) + S_{i+1}(a_{i+1}) f(a_{i+1}), \quad i = 1, \ldots, N-1.
\end{gather*}
Since these equation require unavailable knowledge of $f$ at the points $a_i$ and $b_i$, we impose the following vanishing endpoint conditions on the functions $S_i$:
\be\label{S}
S_i (a_i) = 0 = S_i (b_i), \,\forall i = 1, \ldots, N.
\ee
Thus the requirements on the functions $\lambda_i$ reduce to 
\begin{gather*}
\lambda_1 (x_0) = y_0, \quad \lambda_N (x_N) = y_N\\
\lambda_i(b_i) = y_i = \lambda_{i+1} (a_{i+1}), \quad i = 1, \ldots, N-1.
\end{gather*}
Function tuples $\blambda$ and $\bS$ satisfying \eqref{condS} and \eqref{S} are said to have property (S).

A class of functions $S_i$ for which conditions \eqref{S} hold is, for instance, the class of polynomial B-splines $B_n$ of order $2 < n\in \N$ centered at the midpoint of the interval $[a_i,b_i]$. Polynomial B-splines $B_n$ have even the property that all derivatives up to order $n-2$ vanish at the endpoints: $D^k B_n (a_i) = 0 = D^k B_n (b_i)$, for all $k = 0, 1\ldots, n-2$. 

The above considerations now entail the next theorem.

\begin{theorem}
Let $\X$ and $\X_i$, $i\in \N_N$, be as defined above. Let $\mathscr{I}$ be as in \eqref{Ia}. Suppose that $\blambda, \bS\in  \underset{i=1}{\overset{N}{\times}} C(\X_i)$ and that they have property (S). The RB operator \eqref{eq3.4} maps $C_\mathscr{I}$ as given by \eqref{111} into itself and is well defined. If in addition
\[
\max\left\{\|S_i\|_{\infty,\X_i}\st i\in\N_N\right\} < 1,
\]
then $\Phi$ is contractive on $C_\mathscr{I}$.
\end{theorem}

The fixed point $\ff$ of $\Phi$ is called again a \textbf{continuous local fractal interpolation function}.

\begin{proof}
The assumptions on $\blambda$ and $\bS$ guarantee that $\Phi$ is well defined and maps $C_\mathscr{I}$ into itself. The contractivity of $\Phi$ under the given condition follows immediately from the proof of Theorem \ref{thm7}.
\end{proof}

For the particular setting at hand, one may, of course, also construct fractal functions of class $\dot{C}^s$ and $C^n$ by imposing the relevant conditions on the function tuples $\blambda$ and $\bS$ and choose the appropriate interpolation sets. We rely on the diligent reader to provide these conditions and prove the corresponding results. 
\section{Sobolev Spaces $W^{m,p}$}\label{sec8}
The final type of function space we consider are the Sobolev spaces $W^{m,p}$ with $m\in \N_0$ and $1\leq p \leq \infty$. To this end, let $\Omega\subset\R$ be open and 
$$
C^{m,p}(\Omega) :=\{f\in C^\infty (\Omega)\st D^k f\in L^p(\Omega),\,\forall\,k=0,1,\ldots, m\}.
$$
Define functionals $\|\bullet\|_{m,p}$, $m\in \N_0$ and $1\leq p \leq \infty$, as follows:
\[
\|f\|_{m,p} := \begin{cases}
\displaystyle{\left(\sum_{k=0}^m \|D^k f\|^p_{L^p}\right)^{1/p}}, & 1\leq p < \infty;\\ \\
\displaystyle{\max_{k\in \{0,1,\ldots, m\}}\{\|D^kf\|_\infty\}}, & p = \infty.
\end{cases}
\]
The closure of $C^{m,p}(\Omega)$ in the norm $\|\bullet\|_{m,p}$ produces the Sobolev space $W^{m,p}(\Omega)$. The ordinary derivatives $D^k$ in $C^{m,p}(\Omega)$ have a continuous extension to $W^m(L^p)(\Omega)$. These extensions are then the weak derivatives $D^{(k)}$. The Sobolev space $W^{m,p}(\Omega)$ is a Banach spaces when endowed with the norm $\|\bullet\|_{m,p}$. For more details, we refer the reader to \cite{A}.

Now suppose $X := (0,1)$ and $\{X_i\st i\in \N_N\}$ is a collection of nonempty open intervals of $X$. Further suppose that $\{x_1 < \cdots < x_{N-1}\}$ is a partition of $X$ and that $\{u_i:X_i\to X\}$ is a family of affine mappings with the property that
$u_i(X_i) = (x_{i-1}, x_i)$, for all $i\in \N_N$, where we set $x_0:= 0$ and $x_N:= 1$. We write $W^{m,p}$ for $W^{m,p}(X)$.

\begin{theorem}
Under the assumptions stated above, let $\blambda\in \underset{i=1}{\overset{N}{\times}} W^{m,p}(X_i)$ and let $\bS := (s_1, \ldots, s_N)\in \R^N$. Then the RB operator $\Phi: W^{m,p}\to \R^{(0,1)}$,  $m\in \N_0$ and $1\leq p \leq \infty$, defined by
\[
\Phi g := \sum_{i=1}^N (\lambda_i \circ u_i^{-1})\,\chi_{u_i(X_i)} + \sum_{i=1}^N s_i (g_i\circ u_i^{-1})\,\chi_{u_i(X_i)},
\]
has range contained in $W^{m,p}$ and is well defined. Moreover, if
\be\label{W}
\begin{cases}
\displaystyle{\left(\max_{k\in \{0,1,\ldots,m\}} \sum_{i\in \N_N} \frac{|s_i|^p}{a_i^{kp -1}}\right)^{1/p} < 1}, & 1\leq p < \infty;\\
\displaystyle{\sum_{i\in \N_N} \frac{|s_i|}{a_i^k} < 1}, & p = \infty,
\end{cases}
\ee
then $\Phi$ is contractive on $W^{m,p}$.
\end{theorem}

The unique fixed point $\ff$ of $\Phi$ is called a \textbf{local fractal function of class $W^{m,p}$}.

\begin{proof}
That $\Phi$ is well defined and has range contained in $W^{m,p}$ follows from the assumption on the function tuple $\blambda$ and the fact that if the weak derivative of a function $f$ exits and $u_i$ is a diffeomorphism, then the weak derivative of $f\circ u_i^{-1}$ exists and equals $(D^{(1)}f)(u_i^{-1})\cdot D u_i^{-1}$.

To prove contractivity on $W^{m,p}$, suppose that $g,h\in W^{m,p}$, $k\in \{0,1,\ldots, m\}$. Denote the ordinary derivative of $u_i$ by $a_i$. Note that $a_i > 0$ but may be larger than one. Then, for $1\leq p < \infty$, we obtain the following estimates.
\begin{align*}
\|D^{(k)}\Phi g - D^{(k)}\Phi h\|_{L^p}^p & = \int_X \left\vert D^{(k)}\sum_{i\in \N_N} s_i (g_i - h_i)(u_i^{-1})(x)\right\vert^p \chi_{u_i(X_i)} dx \\
& \leq \sum_{i\in \N_N} |s_i|^p \int_{u_i(X_i)} \left\vert D^{(k)}(g_i - h_i)(u_i^{-1}(x))\right\vert^p\left(\frac{1}{a_i}\right)^{k p} dx\\
& \leq \sum_{i\in \N_N} |s_i|^p \left(\frac{1}{a_i}\right)^{k p-1} \int_{X_i} \left\vert D^{(k)}(g_i - h_i)(x)\right\vert^p dx\\
& \leq \left(\sum_{i\in \N_N} |s_i|^p \left(\frac{1}{a_i}\right)^{k p-1}\right) \|D^{(k)} g - D^{(k)} h\|_{L^p}^p.
\end{align*}
Summing over $k = 0,1,\ldots, m$, and factoring out the maximum value of the expression in parentheses,  proves the statement.

Similarly, for $p=\infty$, we get
\begin{align*}
\left\vert D^{(k)} g (x) - D^{(k)} h (x)\right\vert & = \left\vert\sum_{i\in \N_N} s_i D^{(k)} (g_i - h_i) (u_i^{-1})(x) \left(\frac{1}{a_i^k}\right)\chi_{u_i(X_i)}(x)\right\vert\\
& \leq \sum_{i\in \N_N} \frac{|s_i|}{a_i^k} \left\vert D^{(k)} (g_i - h_i) (u_i^{-1})(x)\chi_{u_i(X_i)}(x)\right\vert\\
& \leq \sum_{i\in \N_N} \frac{|s_i|}{a_i^k} \left\| D^{(k)} g - D^{(k)} h \right\|_{\infty},
\end{align*}
verifying the assertion.
\end{proof}

\section*{Acknowledgment}
The author wishes to thank the Mathematical Sciences Institute of The Australian National University for its kind hospitality and support during his research visit in May 2013 which initiated the investigation into local IFSs.

\end{document}